\documentclass[12pt,letterpaper]{amsart}
\usepackage{amsmath,amsfonts,amssymb, tikz, amsthm}
\usepackage[colorlinks=true,linkcolor=black,anchorcolor=black,citecolor=black,filecolor=black,menucolor=black,runcolor=black,urlcolor=blue]{hyperref}
\usepackage[utf8]{inputenc}
\usepackage[english]{babel}
\usepackage{xcolor}
\usepackage[tableposition=top]{caption} 
\usepackage{epigraph}
\usepackage{mathtools}

\newtheorem*{theorem*}{Theorem}
\newtheorem{theorem}{Theorem}
\newtheorem{lemma}{Lemma}
\newtheorem{proposition}{Proposition}
\newtheorem{corollary}{Corollary}

\newtheorem*{conjecture*}{Conjecture}

\theoremstyle{remark}

\newtheorem*{remark*}{Remark}

\theoremstyle{proof}

\numberwithin{equation}{section}
\numberwithin{assumption}{section}

\Alph{assumption}

\newcommand{\Q}{\mathbb{Q}}
\newcommand{\C}{\mathbb{C}}
\newcommand{\R}{\mathbb{R}}
\newcommand{\N}{\mathbb{N}}

\newcommand{\Gal}{\mbox{Gal}}

\begin{document}
	\title{Moments of non-normal number fields - II}
	
	\author{Krishnarjun Krishnamoorthy}
	\email[Krishnarjun K]{krishnarjunmaths@outlook.com, krishnarjunmaths@gmail.com}
	\address{Yanqi Lake Beijing Institute of Mathemtical Sciences and Applications (BIMSA), No. 544, Hefangkou Village, Huaibei Town, Huairou District, Beijing.}
	
	\keywords{Moments, Dedekind zeta function, Artin $L$ functions, Moments}
	\subjclass[2020]{11F66, 11F30, 11R42, 20C30.}
	
	\maketitle

	\begin{abstract}
		Suppose $K$ is a number field and $a_K(m)$ is the number of integral ideals of norm equal to $m$ in $K$, then for any integer $l$, we asymptotically evaluate the sum
		\[
		\sum_{m\leqslant T} a_K^l(m)
		\]
		as $T\to\infty$. We also consider the moments of the corresponding Dedekind zeta function. We prove lower bounds of expected order of magnitude and slightly improve the known upper bound for the second moment in the non-Galois case.
	\end{abstract}

	\section{Introduction}\label{Section "Introduction"}
	
	Suppose that $K$ is a field extension of degree $d$ over $\Q$ (that is, a number field). Let $a_K(m)$ (for $m\in \N$) denote the number of integral ideals in $K$ of norm equal to $m$. Let $s$ be a complex number. The Dedekind zeta function of $K$ can be expressed as 
	\begin{equation}\label{Equation "Dedekind zeta function definition"}
		\zeta_K(s) := \sum_{m=1}^{\infty} \frac{a_K(m)}{m^s}.
	\end{equation}
	It can be shown that $a_K(m)$ is a multiplicative function and satisfies the bound
	\begin{equation}\label{Equation "Ramanujan bound"}
	a_K(m) \ll_\epsilon m^{\epsilon}
	\end{equation}
	for any positive $\epsilon$. Thus the series \eqref{Equation "Dedekind zeta function definition"} converges absolutely in the half plane $\Re(s) > 1$ where it can be expressed as the following Euler product,
	\begin{equation}
		\zeta_K(s) = \prod_{p}\left(1 + \frac{a_K(p)}{p^s} + \frac{a_K(p^2)}{p^{2s}} + \ldots\right).
	\end{equation}
	Furthermore $\zeta_K(s)$ has a meromorphic continuation to the whole complex plane with a simple pole at $s=1$ and satisfies a functional equation connecting values at $s$ and $1-s$ (see \cite[\S 5, Chapter VII]{Neukirch}).
	
	Dedekind zeta functions are natural generalizations of the Riemann zeta function for number fields. As with the Riemann zeta function, the behavior of $\zeta_K(s)$ inside the critical strip $0 < \Re(s) < 1$ is quite mysterious. There are many aspects of the behavior of Dedekind zeta function inside the critical strip that are of interest. In this paper, we focus on understanding the ``moments" along the critical line (that is $\Re(s) = \frac{1}{2}$)
	\begin{equation}\label{Equation "Continuous moment"}
		I_K^{(l)}(T) :=\int\limits_{1}^{T} \left|\zeta_K\left(\frac{1}{2}+it\right)\right|^l dt.
	\end{equation}
	Obtaining precise asymptotic for the above integral is a very hard problem, and even the base case of $K=\Q$ poses serious difficulties. Thus we turn to the discrete analogue of the above problem, which maybe more accessible. Namely, we ask if we can estimate 
	\begin{equation}\label{Equation "Discrete moment"}
		M_K^{(l)}(T):=\sum_{m\leqslant T} a_K^l(m)
	\end{equation}
	for positive integral values of $l$. The case when $l=1$ is classical and maybe deduced as a consequence of the meromorphic continuation of $\zeta_K(s)$. This is also analogous to obtaining estimates for averages of the higher order divisor function (often called the Piltz divisor problem). This problem was first considered by Chandrasekharan and Narasimhan for the case when $K$ was Galois over $\Q$ and when $l=2$ \cite{ChandNarApprox} and was generalized by Chandrasekharan and Good for arbitrary $l$ \cite{ChanGood} (but still when $K$ was Galois over $\Q$). A particular non-Galois case was settled by Fomenko \cite{Fomenko} and later improved upon by L\"u \cite{Lu}. By different methods, this was further generalized in a recent work of the author (along with Kalyan Chakraborty) for many families of non-Galois number fields \cite{KKJNT}. However, the general problem still remained unsolved. The purpose of this paper is to estimate $M_K^l(T)$ for any number field $K$ and any positive integer $l$ thereby completing the solution to this problem. We also provide a unified treatment which reproduces many of the special cases treated in previous work. The final result conforms to expectations in that the main term is of the order $T$ times a power of $\log(T)$.
	
	Before we state the main theorem, we introduce  and fix the following notation throughout the paper. Every representation that we consider will be over $\C$. Let $K$ be as above and $L$ be its Galois closure. The degree of $K$ shall be denoted by $d$. Denote the Galois groups $Gal(L/\Q)$ as $G$ and its subgroup $Gal(L/K)$ as $H$. Let $1_H$ denote the trivial representation of $H$. Denote the corresponding induction to $G$ as $\rho_H$ and its character as $\chi_{\rho_H}$. We now have the following theorem.
	
	\begin{theorem}\label{Theorem "Discrete moments"}
		Suppose that $l$ is a natural number. There exists an integer $\mathsf{m}_l$ such that 
		\begin{equation}\label{Equation "Discrete Moment"}
			\sum_{m\leqslant T} a_K^l(m) \sim c(l,K) T \log^{\mathsf{m}_l}(T)
		\end{equation}
		for some constant $c(l,K)$ depending on $l$ and $K$, as $T\to \infty$.
		Moreover, we have
		\[
		\mathsf{m}_l = \left(\frac{1}{|G|}\sum_{g\in G} \chi^l_{\rho_H}(g)\right) -1.
		\]
	\end{theorem}		
	As we have mentioned before, it is of interest to find asymptotics for $I_K^{(l)}(T)$ and often such estimates are tied to the estimates for $M_K^{(l)}(T)$. When $K$ is Galois over $\Q$ (equivalently $H$ is the trivial subgroup), lower bounds on $I_K^{(l)}(T)$ of the expected order of magnitude are known unconditionally \cite{AkbaryFodden}. Conditionally on the generalized Riemann hypothesis, upper bounds of the correct order of magnitude (except for an $\epsilon$) are also known \cite{Milinovich}. Below, we provide analogous lower bounds for the non-Galois case. It would be convenient to define
	\begin{equation}
		\beta_K := \left|H\setminus G/H\right|.
	\end{equation}
	
	\begin{theorem}\label{Theorem "Continuous moment"}
		For any rational $k\geqslant 0$, we have
		\[
		\int\limits_{1}^{T} \left|\zeta_K\left(\frac{1}{2}+it\right)\right|^{2k} \gg T\log^{\beta_K k^2}(T).
		\]
	\end{theorem}
	The above lower bound is of the expected order of magnitude. To see this we briefly recall some definitions regarding the Selberg class \cite{SelbergAmalfi} and direct the reader to \cite{RMSelberg} for more details. Define the class $\mathcal{S}$ to consist of Dirichlet series $F(s) := \sum_{n=1}^{\infty} \frac{a_F(n)}{n^s}$ which satisfy the following properties:
	\begin{enumerate}
		\item	(Region of convergence) The series defining $F(s)$ converges absolutely for $\Re(s) > 1$.
		\item 	(Analytic continuation) $F(s)$ extends to a meromorphic function so that for some integer $m\geqslant 0$, $(s-1)^mF(s)$ is an entire function of finite order.
		\item 	(Functional equation) There are numbers $Q > 0, \alpha_i > 0, \Re(r_i) \geqslant 0$ such that
		\[
		\Phi(s) := Q^s \prod_{i=1}^{d} \Gamma(\alpha_is + r_i) F(s)
		\]
		satisfies $\Phi(s) = w\overline{\Phi(1-\overline{s})}$ for some complex number $w$ with $|w|=1$.
		\item 	(Euler product) $F(s)$ can be written as the product $\prod_p F_p(s)$ where $F_p(s) = \exp\left(\sum_{k=1}^{\infty} b_{p^k}/p^{ks}\right)$ where $b_{p^k} = \mathcal{O}(p^{k\theta})$ for some $\theta < 1/2$.
		\item 	(Ramanujan hypothesis) $a_F(n) = \mathcal{O}(n^\epsilon)$ for any fixed $\epsilon > 0$.
	\end{enumerate}
	
	A function $F\in \mathcal{S}$ is called \textit{primitive} if $F$ cannot be written as a product of any two elements of $\mathcal{S}$ except for $F = 1\cdot F$. Selberg made the following conjectures about the elements in $\mathcal{S}$. 
	
	\begin{conjecture*}[Conjecture A]
		For all $F\in \mathcal{S}$, there exists a positive integer $n_F$ such that 
		\[
		\sum_{p\leqslant X} \frac{\left|a_F(p)\right|^2}{p} = n_F\log\log(X) + \mathcal{O}(1).
		\]
	\end{conjecture*}
	
	\begin{conjecture*}[Conjecture B]
		\begin{enumerate}
			\item	For any primitive function $F$, $n_F=1$.
			\item 	For two distinct primitive functions $F, F'$,
			\[
			\sum_{p\leqslant T} \frac{a_F(p)\overline{a_{F'}(p)}}{p} = \mathcal{O}(1).
			\]
		\end{enumerate}
	\end{conjecture*}
	
	It is expected that for an irreducible representation $\xi$ of the Galois group $G$, the Artin $L$ function $L(s,\xi)$ is a primitive element of the Selberg class. If $\{\xi_i\}$ is a complete list of irreducible representations of $G$, then $\zeta_K(s) = \prod_i L(s,\xi_i)^{e_i}$ is a decomposition of $\zeta_K(s)$ into primitive elements (inside the Selberg class), where we have set $e_i := \langle \rho_H, \xi_i\rangle_G$. In this situation \cite[Conjecture 5]{Heapthesis} along with Lemmas \ref{Lemma "Auxiliary 1"} and \ref{Lemma "Auxiliary 2"} leads to the conjecture 
	\begin{equation}\label{Equation "Expectation"}
		\int\limits_{1}^{T} \left|\zeta_K\left(\frac{1}{2}+it\right)\right|^{2k} dt \sim c(k,K) T \log^{\beta_K k^2}(T),
	\end{equation}
	for $k > 0$ and some constant $c(k,K)$ depending on $k$ and $K$.
	Finally, if $K$ is Galois over $\Q$, $H$ will be the trivial subgroup and $\beta_K = d$. The only case of \eqref{Equation "Expectation"} known to be true is when $K$ is a quadratic extension of $\Q$ and when $k=1$ \cite{Motohashi}.
	
	Regarding upper bounds for the moments, very little is known unconditionally. From their approximate functional equation, Chandrasekharan and Narasimhan (\cite[see pg. 61]{ChandNarApprox}) were able to deduce that 
	\begin{equation}\label{Equation "ChandNar"}
		\frac{1}{T}\int\limits_{1}^{T} \left|\zeta_K\left(\frac{1}{2}+it\right)\right|^2 dt = \sum_{m\leqslant cT^{\frac{d}{2}}} \frac{a_K^2(m)}{m} + \mathcal{O}\left(T^{\frac{d}{2}-1}\log^{d}T\right) = \mathcal{O}\left(T^{\frac{d}{2}-1}\log^{d}T\right)
	\end{equation}
	for some constant $c > 0$ and $d > 2$. As a consequence of Theorem \ref{Theorem "Discrete moments"}, we may improve this as follows.
	
	\begin{theorem}\label{Theorem "Improvement"}
		With notation as above, we have
		\begin{equation}
			\frac{1}{T}\int\limits_{1}^{T} \left|\zeta_K\left(\frac{1}{2}+it\right)\right|^2 dt = \sum_{m\leqslant  cT^{\frac{d}{2}}} \frac{a_K^2(m)}{m} + \mathcal{O}\left(T^{\frac{d}{2}-1}\log^{\beta_K}T\right).
		\end{equation}
		In particular,
		\begin{equation}
			\frac{1}{T}\int\limits_{1}^{T} \left|\zeta_K\left(\frac{1}{2}+it\right)\right|^2 dt = \mathcal{O}\left(T^{\frac{d}{2}-1}\log^{\beta_K}T\right)
		\end{equation}
		whenever $d > 2$.
	\end{theorem}
	 The fact that this is indeed an improvement follows from Lemma \ref{Lemma "Group Theory"}.

	\section{Preliminaries}\label{Section "Preliminaries"}
	For the convenience of the reader, we compile some basic facts which we shall use throughout the proofs.
	\subsection{Character theory}
	Given an $n$ dimensional complex representation $\xi$ of a finite group $G$, we denote its character (trace) as $\chi_\xi$. The characters associated to irreducible representations of $G$ form an orthonormal basis for the class functions on $G$ with the inner product defined as 
	\begin{equation}\label{Equation "Definition - Inner Product"}
		\langle f_1, f_2\rangle_G := \frac{1}{|G|} \sum_{g\in G} f_1(g) \overline{f_2(g)}. 
	\end{equation}	
	
	Given two representations $(\xi_1, V_1)$ and $(\xi_2, V_2)$ of a group $G$, we may consider the representation $(\xi_1\otimes \xi_2, V_1\otimes V_2)$ defined as $(\xi_1\otimes \xi_2) (g) (v_1\otimes v_2) = \xi_1(g) (v_1)\otimes \xi_2(g) (v_2)$ for any $v_1\in V_1$ and $v_2\in V_2$ and extended linearly. This is well-defined and satisfies
	\begin{equation}\label{Equation "Character of tensor product"}
		\chi_{\xi_1\otimes \xi_2}(g) = \chi_{\xi_1}(g) \cdot \chi_{\xi_2}(g)
	\end{equation}
	for any $g\in G$. In general the tensor product of two irreducible representations is not irreducible. Understanding the decomposition of tensor products of representations into irreducibles is often referred to as the Clebsch-Gordon problem.
	
	\subsection{Artin $L$ functions}
	
	We start with a number field $ L $, which we shall assume is Galois over $ \Q $. Suppose $G = Gal(L/\Q)$. Let $ v $ (associated with the rational prime $p$) denote a finite place of $ \Q $ and let $ w $ be a place of $ L $ above $ v $. Let $G_w$ and $I_w$ denote the corresponding decomposition and inertia subgroups. We may define an element $\sigma_w$ of $G_w/I_w$ called the Frobenius element at $w$. Except for finitely many places $v$, the inertia subgroup $I_w$ is trivial and thus in those cases $\sigma_w$ is an element of the Galois group $G$. In any case, as $w$ runs through the places over $v$, the corresponding Frobenius elements (defined modulo inertia) are conjugates of one another. Thus, by abuse of notation, we shall consider the Frobenius at $v$ (or $p$) and denote it by $\sigma_v$ (or $\sigma_p$). This is justified because we shall be primarily interested in functions of $\sigma_w$ which are invariant under conjugation (such as trace).
	
	Suppose that $\xi : G \to Aut(V)$ is a representation over a (finite dimensional) complex vector space $V$. For every $w$, $\xi$ maybe considered as a representation of $G_w$ on $V$ and thus yields a representation of $G_w/I_w$ on the fixed subspace $V^{I_w}$. The Artin $ L $ function attached to the representation $ \xi $ is defined by 
	\begin{equation}\label{Equation "Artin L function definition"}
		L(s,\xi) := \prod_{v < \infty} \frac{1}{\det\left(\left(Id - p^{-s}\xi(\sigma_w))\right| V^{I_w}\right)} =: \prod_{p} L_p(\xi,s).
	\end{equation}
	The product is absolutely convergent for $ \Re(s) > 1 $. We collect some of the important properties of the Artin $ L $ function for future reference.
	
	\begin{proposition}\label{Proposition "Artin L function Properties"}
		The Artin $ L $ functions defined above has the following properties.
		\begin{enumerate}
			\item	$ L(s, \xi_1\oplus\xi_2) = L(s,\xi_1)L(s,\xi_2) $ for any two representations $ \xi_1 $ and $ \xi_2 $.
			\item	Suppose $ \Q\subset K\subset L $ is an intermediate field which is Galois over $ \Q $. Let $ H = \Gal\left(L/K\right) $. Then a representation of $ \xi $ of $ G/H $ may be lifted to a representation $ \tilde{\xi} $ via the canonical projection $ G\to G/H $. Then $ L(s,\xi) = L(s,\tilde{\xi}) $, where the first $ L $ function is considered in the setting of $ L $ over $ \Q $ and the second $ L $ function is considered in the setting $ K $ over $ \Q $.
			\item	Suppose $ K $ is an intermediate field, not necessarily Galois over $ \Q $. Let $ H $ denote the Galois group of $ L $ over $ K $. For a representation $ \xi $ of $ H $, we have $ L(s, \xi) = L(s, \mbox{Ind}_{H}^G \xi) $, where $ \mbox{Ind}_{H}^G \xi $ denotes the representation induced from $ H $ to $ G $.
			\item 	With notation as in the previous statement. Then 
			\[
			\zeta_{K}(s) = L(s, \mbox{Ind}_{H}^G 1_{H})
			\]
			where $1_{H}$ is the trivial representation of $H$.
		\end{enumerate}
	\end{proposition}

	\begin{remark*}
		In the sequel, we shall use Proposition \ref{Proposition "Artin L function Properties"} repeatedly at various steps without referring back to it every time.
	\end{remark*}

%	\begin{lemma}\label{Lemma "Convexity bound"}
%		Suppose the strong Artin conjecture is true. Let $\xi$ be an $n$-dimensional complex representation of $Gal(L/\Q)$ for some Galois extension $L$ of $\Q$. Then for any $\frac{1}{2}\leqslant \sigma \leqslant 1$ and $t\in \R$
%		\[
%		\left|L\left(\sigma +it, \xi\right)\right| \ll_\epsilon (1+ |t|)^{\frac{n}{2} \sigma + \epsilon}.
%		\]
%		where $\epsilon > 0$ is an arbitrarily small constant.
%	\end{lemma}
%	
%	\begin{proof}
%		It is sufficient to prove the above lemma for irrecudible representations $\xi$. Hence, without loss of generality we may assume that $\xi$ is irreducible. The strong Artin conjecture \q\ states that there exists an automorphic representation $\pi$ of $GL_n(A_\Q)$ such that $L(s,\xi) = L(s,\pi)$. The degree of $L(s,\pi)$ is $n$ and one may directly apply the convexity bound \q\ in this case to complete the proof.
%	\end{proof}
	
%	\textcolor{blue}{State the strong Artin conjecture and mention what are the cases which are known. Further, state what are the convexity bounds in those cases.}
	
	For an irreducible non-trivial representation $\xi$ of $G$, the Artin holomorphy conjecture asserts that $L(s,\xi)$ continues holomorphically to the whole complex plane. This is unknown at the moment but we have a very general partial result. Brauer's induction theorem is an important result in representation theory of groups which is particularly consequential in the study Artin $L$ functions. Suppose that $\xi$ is an irreducible non-trivial representation of $G$. Brauer's theorem establishes the meromorphic continuation of $L(s,\xi)$ to the whole complex plane. The following slightly stronger consequence of Brauer's theorem shall be useful for us in the sequel (see \cite[Corollary 5.47]{IwaniecKowalski}).
	
	\begin{lemma}\label{Lemma "Artin"}
		If $\xi$ is a non-trivial irreducible representation of $G$, then $L(s,\xi)$ has neither zeros nor poles in the region $\Re(s) \geqslant 1$.
	\end{lemma}

	\section{Proof of Theorem \ref{Theorem "Discrete moments"}}
	In the remainder of this paper, it is convenient, in many instances, to restrict ourselves to unramified primes. As we are omitting only finitely many primes, this does not affect the nature of our results but for the exact constants. The order of growth shall remain the same. The strategy of the proof is essentially that of \cite{KKJNT} adopted and generalized with modifications for the current needs.
	Define
	\begin{equation}\label{Equation "D_l(s) definition"}
		D_l(s) := \sum_{m=1}^{\infty} \frac{a_K^l(m)}{m^s} = \prod_{p}\left(1 + \frac{a_K^l(p)}{p^s} +\frac{a_K^l(p^2)}{p^{2s}} \ldots\right).
	\end{equation}
	From \eqref{Equation "Ramanujan bound"}, $D_l(s)$ is absolutely convergent for $\Re(s) > 1$ where the Euler product expression is valid. 
	
	\subsection{Meromorphic Continuation of $D_l(s)$}
	
	We wish to establish a meromorphic continuation of $D_l(s)$ to a larger domain.
	As $\zeta_K(s) = L(s, \rho_H)$, for every unramified prime $p$, we have 
	\[
	a_K(p) = \chi_{\rho_H}(\sigma_p)
	\]
	where $\sigma_p$ is a choice of Frobenius at $p$. From \eqref{Equation "Character of tensor product"}, it follows that
	\[
	a_K^l(p) = \chi_{\rho_H^{\otimes l}}(\sigma_p),
	\]
	where
	\[
	\rho_H^{\otimes l} = \underbrace{\rho_H\otimes \rho_H\otimes \ldots\otimes \rho_H}_{l \mbox{\tiny{ times}}}.
	\]
	Therefore from definition,
	\[
	L(s, \rho_H^{\otimes l}) = E(s) \prod_{p\mbox{\tiny{ unramified}}} \left(1 - \frac{a_K^l(p)}{p^s} + \ldots\right)^{-1},
	\]
	where $E(s)$ is the product of Euler factors at the ramified primes. Furthermore, for $\Re(s) > 1$, we have
	\begin{align}\nonumber
		L(s, \rho_H^{\otimes l}) &=\prod_{p} \left(1 - \frac{a_K^l(p)}{p^s}\right)^{-1} U_1(s)\\ \label{Equation "U_1"}
		&= \prod_{p} \left(1 + \frac{a_K^l(p)}{p^s} + \frac{a_K^{2l}(p)}{p^{2s}} + \ldots\right) U_1(s),
	\end{align}
	where $U_1(s)$ is holomorphic in the region $\Re(s) > \frac{1}{2}$. Comparing Euler factors with \eqref{Equation "D_l(s) definition"}, we see that 
	\begin{equation}\label{Equation "D_l(s) analytic continuation"}
		D_l(s) = L(s, \rho_H^{\otimes l}) U_2(s)
	\end{equation}
	where $U_2(s)$ is again a function holomorphic in the region $\Re(s) > \frac{1}{2}$. The region of holomorphy of $U_1(s)$ and $U_2(s)$ can be deduced using \eqref{Equation "Ramanujan bound"}. In particular, from the meromorphic continuation of $L(s, \rho_H^{\otimes l})$ to the whole complex plane, we may conclude that $D_l(s)$ continues meromorphically to the region $\Re(s) > \frac{1}{2}$.
	
	\subsection{Completing the proof}
	
	Let $\{\xi\}_{i=1}^{n}$ denote the complete set of irreducible representations of $G$ with $\xi_1$ being the trivial representation. Let
	\begin{equation}\label{Equation "chi decomposition"}
		\chi_{\rho_H^{\otimes l}} = \sum_{i=1}^{n} \mathsf{m}_i^{(l)} \chi_{\xi_i}
	\end{equation}
	denote the decomposition of $\rho_H^{\otimes l}$ into irreducible representations. Translating this into Artin $L$ functions gives us 
	\[
	L(s, \rho_H^{\otimes l})  = \prod_{i=1}^{n} L^{\mathsf{m}_i^{(l)}}(s, \xi_i) = \zeta^{\mathsf{m}_1^{(l)}}(s)\prod_{i=2}^{n} L^{\mathsf{m}_i^{(l)}}(s, \xi_i).
	\]
	Thus $L(s, \rho_H^{\otimes l})$ has a pole of order $\mathsf{m}_1^{(l)}$ at the point $s=1$ and is otherwise continuous on the half plane $\Re(s) \geqslant 1$ (from Lemma \ref{Lemma "Artin"}). Hence from \eqref{Equation "D_l(s) analytic continuation"}, $D_l(s)$ is a continuous function in the region $\Re(s) \geqslant 1$ but for a pole of order $\mathsf{m}_1^{(l)}$ at the point $s=1$. Now we apply the Delange-Ikehara Tauberian theorem (see \cite[Corollary, Pg. 121]{NarkiewiczNT}) to the Dirichlet series $D_l(s)$ and get
	\[
	M_K^{(l)}(T) \sim c T\log^{\mathsf{m}_1^{(l)}-1}(T)
	\]
	for some constant $c$. In fact, $c$ maybe expressed in terms of the leading coefficient in the Laurent series expansion of $D_l(s)$ about the point $s=1$.
	Finally from definition, we have 
	\begin{equation}\label{Equation "m_1^l formula"}
		\mathsf{m}_1^{(l)} = \langle \rho_H^{\otimes l}, 1_G\rangle_G = \frac{1}{|G|} \sum_{g\in G} \chi_{\rho_H}^l(g).
	\end{equation}	
	Setting $\mathsf{m}_l = \mathsf{m}_1^{(l)}-1$ completes the proof.
	
	\section{Proof of Theorem \ref{Theorem "Continuous moment"}}
	
	We first note two applications of the Chebotarev density theorem.

	\begin{lemma}\label{Lemma "CDT 1"}
		For any two representations $\rho_1, \rho_2$ of $G$, 
		\[
		\underset{p \mbox{ \tiny{unramified}}}{\sum_{p\leqslant T}} \frac{\chi_{\rho_1}(\sigma_p) \overline{\chi_{\rho_2}(\sigma_p)}}{p} = \langle \rho_1, \rho_2\rangle_G \log\log(T) + \mathcal{O}(1).
		\]
	\end{lemma}
	
	\begin{proof}
		We shall restrict ourselves to unramified primes throughout the proof and remove it from notation. From the Chebotarev density theorem, for any conjugacy class $C$ of $G$, 
		\[
		\underset{\sigma_p\in C}{\sum_{p\leqslant T}}\frac{1}{p} = \frac{|C|}{|G|} \log\log(T) + \mathcal{O}(1).
		\]
		Therefore,
		\begin{align*}
			\sum_{p\leqslant T} \frac{\chi_{\rho_1}(\sigma_p) \overline{\chi_{\rho_2}(\sigma_p)}}{p} &= \sum_{C} \underset{\sigma_p\in C}{\sum_{p\leqslant T}} \frac{\chi_{\rho_1}(\sigma_p) \overline{\chi_{\rho_2}(\sigma_p)}}{p} \\
			&= \sum_{C} \chi_{\rho_1}(g_C) \overline{\chi_{\rho_2}(g_C)} \underset{\sigma_p\in C}{\sum_{p\leqslant T}} \frac{1}{p} \\
			&= \frac{1}{|G|} \sum_{C} \chi_{\rho_1}(g_C) \overline{\chi_{\rho_2}(g_C)} |C| \log\log(T) + \mathcal{O}(1)\\
			&= \left(\frac{1}{|G|} \sum_{g\in G} \chi_{\rho_1}(g) \overline{\chi_{\rho_2}(g)}\right) \log\log(T) + \mathcal{O}(1)\\
			&= \langle \rho_1, \rho_2\rangle_G \log\log(T) + \mathcal{O}(1).
		\end{align*}
		Here $\sum_{C}$ denotes the sum over the various conjugacy classes of $G$ and $g_C$ denotes an arbitrary element in each conjugacy class. This completes the proof.
	\end{proof}
	
	\begin{corollary}\label{Corollary "CDT 1"}
		With notation as above
		\[
		\sum_{n\leqslant T} \frac{a_K^l(p)}{p} = (\mathsf{m}_l+1) \log\log(T) + \mathcal{O}(1).
		\]
	\end{corollary}

	\begin{proof}
		We may replace the sum over all the primes with the sum over the unramified primes. We note that, for an unramified prime $p$, $a_K^l(p) = \chi_{\rho_H^{\otimes l}}(\sigma_p) \cdot \chi_{1_G}(\sigma_p)$. Therefore from the previous lemma, 
		\[
		\sum_{p\leqslant T} \frac{a_K^l(p)}{p} = \langle \rho_H^{\otimes l}, 1_G\rangle_G \log\log(T) + \mathcal{O}(1).
		\]
		From definition $\mathsf{m}_1^{(l)} = \langle \rho_H^{\otimes l}, 1_G\rangle_G = \mathsf{m}_l+1$. This completes the proof.
	\end{proof}
	
	The proof of the following lemma and corollary follow almost verbatim to the proofs above and hence we omit the details. We use the Chebatorev density theorem in the following form
	\[
	\underset{\sigma_p\in C,\ p \mbox{ \tiny{unramified}}}{\sum_{p\leqslant T}} 1 \sim \frac{|C|}{|G|} \frac{T}{\log(T)}.
	\]
	
	\begin{lemma}\label{Lemma "CDT 2"}
		For any two representations $\rho_1, \rho_2$ of $G$, 
		\[
		\underset{p \mbox{ \tiny{unramified}}}{\sum_{p\leqslant T}} \chi_{\rho_1}(\sigma_p) \overline{\chi_{\rho_2}(\sigma_p)} \sim \langle \rho_1, \rho_2\rangle_G \frac{T}{\log(T)}.
		\]
	\end{lemma}
	
	\begin{corollary}
		With notation as above
		\[
		\sum_{p\leqslant T} a_K^l(p) \sim (\mathsf{m}_l+1) \frac{T}{\log(T)}
		\]
	\end{corollary}
	
	\begin{lemma}\label{Lemma "Auxiliary 1"}
		With notation as above
		\[
		\mathsf{m}_1^{(2)} = \beta_K.
		\]
	\end{lemma}
	
	\begin{proof}
		Observe that
		\[
		\langle\rho_H^{\otimes 2}, 1_G\rangle_G = \frac{1}{|G|} \sum_{g\in G} \chi_{\rho_H}^2(g) = \langle \rho_H, \overline{\rho_H}\rangle_G.
		\]
		From Frobenius reciprocity
		\begin{align*}
			\langle\rho_H, \overline{\rho_H}\rangle_G	&= \langle Ind(1_H), \overline{Ind (1_H)}\rangle_G\\
			&= \left\langle 1_H, Res \left(\overline{Ind (1_H)}\right)\right\rangle_H.
		\end{align*}
		From Mackey's theorem, we have 
		\[
		Res \left(\overline{Ind (1_H)}\right) = \overline{Res \left(Ind (1_H)\right)} = \bigoplus_{g\in H\setminus G/H} \overline{Ind_{H_g}^H 1_{H_g}},
		\]
		where $H_g := gHg^{-1}\cap H$. Therefore,
		\[
		\left\langle1_H, Res \left(\overline{Ind (1_H)}\right)\right\rangle_H = \sum_{g\in H\setminus G/H} \langle 1_H, \overline{Ind_{H_g}^H 1_{H_g}}\rangle_H.
		\]
		The induced representation of $1_{H_g}$ contains precisely one copy of $1_H$ for every $g$ and hence $\langle \rho_H^{\otimes 2}, 1_G\rangle_G = \left|H \setminus G / H\right| = \beta_K $ completing the proof.
	\end{proof}
	
	\begin{proof}[Proof of Theorem \ref{Theorem "Continuous moment"}]
		Our strategy is to apply \cite[Theorem 2.3]{AkbaryFodden} to the Dedekind zeta function $\zeta_K(s)$. In their notation, $\zeta_K(s)$ satisfies the necessary conditions of their theorem with $\beta = \mathsf{m}_2 + 1 = \mathsf{m}_1^{(2)}$. This follows from Lemmas \ref{Lemma "CDT 1"} and \ref{Lemma "CDT 2"} above. From Lemma \ref{Lemma "Auxiliary 1"}, $\mathsf{m}_1^{(2)} = \beta = \beta_K$ and we are done.
	\end{proof}
	
	\section{Proof of Theorem \ref{Theorem "Improvement"}}
	
	The proof follows along the same lines as \cite[\S 6]{ChandNarApprox}, where we use Theorem \ref{Theorem "Discrete moments"} in place of their \cite[Theorem 3]{ChandNarApprox}. The key improvement of the bound on the off-diagonal terms is given in the following lemma. First we observe from Theorem \ref{Theorem "Discrete moments"} and the results in the previous section, that 
	\begin{equation}\label{Equation "Special case"}
		\sum_{m\leqslant T} a_K^2(m)\ll T\log^{\beta_K}(T).
	\end{equation}
	\begin{lemma}\label{Lemma "ChandNar"}
		Define
		\[
		S(T) := \sum_{m < n\leqslant T} \frac{a_K(m)a_K(n)}{\sqrt{mn} \log\left(\frac{n}{m}\right)}.
		\]
		Then $S(T) = \mathcal{O}\left(T \log^{\beta_K}(T)\right)$.
	\end{lemma}
	
	\begin{proof}
		In the range $n > m$, we have $\log(\frac{n}{m}) > 1- \frac{m}{n}$. We therefore have
		\[
		S(T) < \sum_{m < n\leqslant T} \frac{a_K(m)a_K(n)}{\sqrt{mn}} + \sum_{m < n\leqslant T} \frac{a_K(m)a_K(n)}{n-m}.
		\]
		As in \cite[proof of Lemma 11]{ChandNarApprox}, the first term is $\mathcal{O}(T)$ and the second term is bounded above by $\sum_{m\leqslant T} a_K^2(m)$. The lemma now follows from \eqref{Equation "Special case"}.
	\end{proof}
	Now we may follow \cite[\S 6]{ChandNarApprox} almost verbatim and complete the proof of Theorem \ref{Theorem "Improvement"}. The following lemma applied to $G = Gal(L/\Q)$ and $H = Gal(L/K)$ implies that Theorem \ref{Theorem "Improvement"} is indeed an improvement over \eqref{Equation "ChandNar"}.
	\begin{lemma}\label{Lemma "Group Theory"}
		If $G$ is a group and $H$ is a subgroup of $G$, then $|H\setminus G/H| = |G/H|$ if and only if $H$ is a normal subgroup of $G$.
	\end{lemma}
	
	\begin{proof}
		Consider the action of $H$ on the left on the set of left cosets $G/H$ (the action of $h$ maps $gH$ to $(hg)H$). Then the set of double cosets is in bijection with the orbits under this action ($HgH\mapsto$ orbit of $gH$). It can be shown that the size of the orbit containing $HgH$ under this action is the index $[H : H\cap gHg^{-1}]$. If $H$ were normal in $G$, then each of the above index is $1$ and hence $|H\setminus G/H| = |G/H|$. If $H$ were not normal in $G$, then there exists a $g\in G$ such that $H\neq gHg^{-1}$. It follows that the size of the orbit containing $gH$ has at least two elements and thus the number of orbits (which equals $|H\setminus G/H|$) should be strictly less than the size of the set on which $H$ is acting (which is $|G/H|$).
	\end{proof}

	\section{Further remarks}\label{Section "Further Remarks"}
	
	We first have the following lemma.
	\begin{lemma}\label{Lemma "Auxiliary 2"}
		If
		\[
		\rho_H = \bigoplus_{i=1}^{n} \xi_i^{e_i}
		\]
		is the decomposition of $\rho_H$ into irreducible representations, then 
		\[
		\mathsf{m}_1^{(2)} = \sum_{i} e_i^2.
		\]
	\end{lemma}
	
	\begin{proof}
		Proceeding as in the previous proof, we have $\mathsf{m}_1^{(2)} = \langle \rho_H, \overline{\rho_H}\rangle_G$. But
		\[
		\chi_{\rho_H}(g) = \frac{1}{|H|} \sum_{r\in G} \tau_H(r^{-1}gr) \in \R.
		\]
		Here $\tau_H$ is the characteristic function of $H$. Hence $\rho_H = \overline{\rho_H}$ giving us $\mathsf{m}_1^{(2)} = \langle \rho_H, \rho_H\rangle_G = \sum_{i =1}^ne_i^2$. This completes the proof.
	\end{proof}
	
	As mentioned above, from the general conjectures regarding the primitivity of Artin $L$ functions, $\zeta_K(s) = \prod_{i=1}^{n} L(s,\xi_i)^{e_i}$ is a decomposition of $\zeta_K$ into primitive elements of $\mathcal{S}$ (where as before $\{\xi_i\}$ is a complete set of irreducible representations of $G$). If this is the case, Lemma \ref{Lemma "Auxiliary 2"} would follow from Selberg's Conjecture B (see \cite[Proposition 2.5(a)]{RMSelberg} and \cite[Corollary 3.2]{RMSelberg}). In any case the results of this paper would follow if we assume Selberg's conjectures. As was shown in \cite{RMSelberg}, Selberg's conjectures imply the Langlands reciprocity conjecture (or the strong Artin conjecture) from where we may show the required asymptotic in Theorem \ref{Theorem "Discrete moments"} with a power saving error term. In fact it is expected that the right hand side of \eqref{Equation "Discrete Moment"} should be of the form $TP_l(\log(T)) + \mathcal{O}(T^{1-\theta})$ for some polynomial $P_l$ of degree $\mathsf{m}_l$ and for some $\theta > 0$ (as in the main theorem of  \cite{KKJNT}).
	
	\section*{Acknowledgement}
	
	The author wishes to thank Prof. Sudhir Pujahari for his encouragement.
	
	\bibliographystyle{amsalpha} 
	
	\bibliography{Bibliography}	
	 
\end{document}